\documentclass[11pt, oneside]{article}   	
\usepackage{geometry}                		
\usepackage{graphicx}				
\usepackage{amssymb}
\usepackage{amsmath}
\usepackage{amsthm}
\usepackage{mathrsfs}
\usepackage{enumerate}
\usepackage{xcolor}
\title{Energy decay for the time dependent damped wave equation}
\author{Perry Kleinhenz}
\date{}							

\theoremstyle{definition}
\newtheorem{definition}{Definition}

\newtheorem{assumption}{Assumption}
\newtheorem{remark}{Remark}
\theoremstyle{theorem}
\newtheorem{theorem}{Theorem}[section]
\newtheorem{lemma}[theorem]{Lemma}

\newtheorem{proposition}[theorem]{Proposition}
\newcommand{\inter}{\text{int}}

\newcommand{\Rb}{\mathbb{R}}

\newcommand{\Nb}{\mathbb{N}}

\newcommand{\ra}{\rightarrow}

\newcommand{\rhu}{\rightharpoonup}
\newcommand{\<}{\left\langle}
\renewcommand{\>}{\right\rangle}
\newcommand{\e}{\varepsilon}
\renewcommand{\d}{\delta}

\newcommand{\limn}{\lim_{n \ra \infty}}

\newcommand{\limj}{\lim_{j \ra \infty}}
\newcommand{\nm}[1]{\| #1 \|}

\newcommand{\lp}[2]{ \nm{#1}_{L^{#2}}}

\newcommand{\hp}[2]{\nm{#1}_{H^{#2}}}

\newcommand{\ltwo}[1]{\lp{#1}{2}}
\newcommand{\ltwoo}[1]{\nm{ #1}_{L^2_{t_0}}}
\newcommand{\ltwooj}[1]{\nm{ #1}_{L^2_{t_j}}}
\newcommand{\ltwooo}[1]{\nm{ #1}_{L^2_0}}

\newcommand{\ltwom}[1]{\| #1\|_{L^2(M)}}

\newcommand{\vphi}{\varphi}

\newcommand{\p}{\partial}

\newcommand{\ti}{\widetilde}

\newcommand{\Cm}{\overline{C}}

\newcommand{\Omegaa}{\Omega_{t_0}}

\newcommand{\Winf}{W_{\infty}}

\newcommand{\Omegab}{\overline{\Omega}}
\newcommand{\Char}{\text{Char}}

\def\XXint#1#2#3{{\setbox0=\hbox{$#1{#2#3}{\int}$ }
\vcenter{\hbox{$#2#3$ }}\kern-.6\wd0}}

\begin{document}
\maketitle
\begin{abstract}
Energy decay is established for the damped wave equation on compact Riemannian manifolds where the damping coefficient is allowed to depend on time. Using a time dependent observability inequality, it is shown that the energy of solutions decays at an exponential rate if the damping coefficient 
satisfies a time dependent analogue of the classical geometric control condition. Existing time dependent observability inequalities are improved by removing technical assumptions on the permitted initial data. 
\end{abstract}
\section{Introduction}
Let $(M,g)$ be a smooth compact manifold and let $\Delta_g$ be the associated Laplace-Beltrami operator. Let $W \in L^{\infty} (M\times [0,\infty))$ be a nonnegative function. Consider the damped wave equation with time dependent damping
\begin{equation}\label{DWE}
\begin{cases}
(\p_t^2 -\Delta_g + W(x,t)\p_t)u=0, & (x,t) \in M \times (0,\infty) \\
(u,\p_t u)|_{t=0} =(u_0, u_1) & \in H^1(M) \times L^2(M).
\end{cases}
\end{equation}
The standard object of study is the energy of the solution 
$$
E(u,t) = \frac{1}{2} \int_M |\nabla_g u(x,t)|^2 + |\p_t u(x,t)|^2 dx_g.
$$
It is straightforward to compute 
$$
\frac{d}{dt} E(u,t) = -\int W(x,t) |\p_t u(x,t)|^2 dx_g \leq 0,
$$
where the sign is guaranteed by $W(x,t) \geq 0$. Because of this the energy is non-increasing, but there is no indication of a decay rate as $t \ra \infty$. 

The most straightforward type of decay is uniform stabilization. That is, the existence of a function $r(t)\ra 0$ as $t \ra \infty$, such that 
$$
E(u,t) \leq r(t) E(u,0).
$$

When $W$ does not depend on time, uniform stabilization is equivalent to $W$ satisfying the Geometric Control Condition (GCC) \cite{Ralston1969, RauchTaylor1975}. The GCC is satisfied if there exists some $L>0,$ such that every geodesic with length at least $L$ intersects the set $\{W>0\}$. When the damping is time independent, solutions are a semigroup, so when uniform stabilization occurs it does so with $r(t)=Ce^{-ct}$ for some $C,c>0$. 

There is an equivalent condition to the geometric control condition introduced in \cite{Lebeau1996}. The GCC is equivalent to the existence of $T_0,\Cm>0,$ such that for all unit speed geodesics $\gamma(t)$ and $T \geq T_0$
$$
\frac{1}{T}\int_0^T W(\gamma(t)) dt \geq \Cm.
$$
That is, there is a uniform lower bound on the long time average of the damping along any geodesic. 

In this paper, I show that the appropriate generalization of this condition to the time dependent setting implies exponential decay. 

\begin{assumption}\label{TGCC} (Time-dependent geometric control condition)
Assume there exists $T_0, \Cm >0,$ such that for all unit speed geodesics $\gamma(t)$, any starting time $t_0 \in [0,\infty),$ and $T \geq T_0$
$$
\frac{1}{T} \int_0^T W(\gamma(t), t_0 + t) dt \geq \Cm.
$$
\end{assumption}
\begin{theorem}\label{mainresult}
Suppose $W(x,t) \in C^0_u(M \times [0,\infty))$, that is $W$ is uniformly continuous and uniformly bounded. If $W$ satisfies Assumption \ref{TGCC}, then there exists $C,c>0,$ such that all solutions to \eqref{DWE} satisfy
$$
E(u,t+t_0) \leq C e^{-ct} E(u,t_0), \quad t_0, t \geq 0.
$$
\end{theorem}

In fact, this result holds when $M$ is replaced by an open subset $\Omega$, potentially with a boundary. Let $\Omega$ be an open bounded connected subset of $M$, with a smooth boundary if $\p \Omega \neq \emptyset$. When $\p \Omega$ is nonempty let $B$ either be the Dirichlet trace operator, $B u = u |_{\p \Omega},$ or the Neumann trace operator, $Bu = \p_n u|_{\p \Omega}$,
where $\p_n$ is the outward normal derivative on $\p \Omega$.

Let $A=-\Delta_g$ be the Laplace operator with domain
$$
D(A) = \{u \in L^2(\Omega); Au \in L^2(\Omega) \text{ and } Bu=0 \text{ when } \p\Omega \neq \emptyset\}.
$$
Note $A$ is self adjoint and nonnegative. 
With Dirichlet boundary conditions define $H = H^1_0(\Omega)$. With Neumann boundary conditions, or when $\p\Omega=\emptyset$, let $H=H^1(\Omega)$. 

Let $W \in C_u^0( \overline{\Omega} \times (0,\infty))$ and consider the damped wave equation 
\begin{equation}\label{TDWEb}
\begin{cases}
(\p_t^2 + A+ W(x,t) \p_t ) u =0,&(x,t) \in \Omega \times (0,\infty)\\
Bu = 0 &\text{when } \p\Omega \neq \emptyset  \\
(u,\p_t u)|_{t=0} = (u_0, u_1) &\in H \times L^2(\Omega).
\end{cases}
\end{equation}
Note that for any $(u_0, u_1) \in H \times L^2(\Omega)$ there is a unique weak solution $u \in L^2((0,\infty); H)$ of \eqref{TDWEb} with $\p_t u \in L^2((0,\infty); L^2(\Omega))$.

When there is a boundary, the appropriate generalization of the GCC uses generalized geodesics. See Appendix \ref{nullbichar} for details on the construction of these. 
\begin{assumption}\label{TGCCp}
Assume there exists $T_0, \Cm >0,$ such that for any unit speed generalized geodesic $\gamma(t)$, any starting time $t_0 \in [0,\infty),$ and $T \geq T_0$
$$
\frac{1}{T} \int_0^T W(\gamma(t), t_0 + t) dt \geq \Cm.
$$
\end{assumption}

\begin{theorem}
Suppose $W(x,t) \in C^0_u(\overline{\Omega} \times [0, \infty))$, that is $W$ is uniformly continuous and uniformly bounded, up to the boundary if there is one. When $\p\Omega \neq \emptyset$, assume that no generalized geodesic has contact of infinite order with $(0,T) \times \p\Omega$, that is $G^{\infty}=\emptyset$. If $W$ satisfies Assumption \ref{TGCCp}, then there exists $C,c>0,$ such that all solutions to \eqref{TDWEb} satisfy
$$
E(u,t) \leq C e^{-ct} E(u,0), \quad t \geq 0.
$$
\end{theorem}
See Appendix \ref{nullbichar} for a precise definition of $G^{\infty}$. 

When $\Omega$ is compact without boundary and $W \in C^0(\Omega)$, \cite{RauchTaylor1975} show that the GCC implies exponential decay. The techniques of \cite{Ralston1969} can be used to see that without the GCC exponential decay cannot occur. The best possible exponential decay rate was computed in \cite{Lebeau1996} in terms of long time averages of the damping and the spectral abscissa of the stationary equation. When $\Omega$ has boundary and $W \in C^0(\Omega)$, \cite{BardosLebeauRauch1992} proved that the GCC implies exponential decay, while \cite{BurqGerard1997} show that the GCC is necessary for exponential decay. When $\Omega$ is compact without boundary and $W$ is a 0th order pseudodifferential operator \cite{KeelerKleinhenz2023} show that a generalization of the GCC is equivalent to exponential decay and compute the sharp exponential decay rate. 



There are a variety of results when the damping is allowed to depend on time, although many of them apply only in Euclidean space, require the damping to never vanish or establish polynomial decay rates for damping tending to 0 as time goes to infinity \cite{Mochizuki1976, Matsumura1977, Uesaka1980, MochizukiNakazawa2001, Wirth2004, Wirth2006, Wirth2007, HirosawaWirth2008, wirthperiodic, Kenigson2011, PaunonenSeifert2019, vjdl}.
There are fewer energy decay results on manifolds when the damping term is allowed to vanish. Study of such damping using microlocal methods, goes back to \cite{RauchTaylor1975b}, where they show, if growing eigenmodes can be ruled out, exponential decay of energy holds for time periodic damping satisfying a GCC hypothesis. 
Using a different approach \cite{LRLTT} show that time periodic $W$ satisfying a slightly different GCC hypothesis (Assumption \ref{as2} in this paper) implies exponential decay and discuss several explicit examples. Note that any time periodic damping which satisfies the GCC hypotheses of \cite{LRLTT} also satisfies Assumption \ref{TGCCp}. For further discussion of time periodic damping see \cite{wirthperiodic, PaunonenSeifert2019}. This paper, and indeed \cite{BardosLebeauRauch1992} and \cite{LRLTT}, prove exponential stability for the damped wave equation from observability of the standard wave equation from $\{W>\e\}$, this approach goes back to \cite{Haraux1989}. For related work on time dependent observability see \cite{Shao2019}. 

\textbf{Acknowledgements} I would like to thank Andr\'as Vasy for proposing this question to me and for his helpful comments throughout the course of this project. I would also like to thank the anonymous referee of the first version of this paper for pointing out errors in that version, as well as the anonymous referee of the second version for comments that improved the exposition. I am also thankful to Jared Wunsch, Ruoyu P.T. Wang, and Willie Wong for helpful conversations. I would also like to thank Emmanuel Tr\'elat for helpful correspondence.

\section{Outline of Proof}\label{intermediatesection}
Given $t_0 \in [0, \infty)$ and a $T>0$, to be specified, define $\Omegaa=\Omega \times (t_0,t_0+T)$. For ease of notation, norms with a subscript $t_0$ are taken over $\Omegaa$ for example 
$$
\ltwoo{u}^2 = \int_{t_0}^{t_0+T} \int_{\Omega} |u(x,t)|^2 dx dt, \quad \text{ and  } \quad \ltwooo{u}^2 = \int_0^T \int_\Omega |u|^2 dx dt.
$$

The following observability type result for the damped wave equation is used to prove the main theorem. 
\begin{proposition}\label{propofsing} If $W \in C_b^{0}(\Omegab \times [0,\infty))$ satisfies Assumption \ref{TGCCp} and $T>T_0$, then there exists $C_1>0,$ such that for all $t_0 \in [0,\infty)$ and all solutions $u$ of \eqref{TDWEb}
$$
C_1 E(u,t_0)  \leq  \ltwoo{W^{\frac{1}{2}} \p_t u}^2.
$$
\end{proposition}
Proposition \ref{propofsing} is proved in Section \ref{propsection}.  

\begin{remark}
The key feature in Proposition \ref{propofsing} is that the constant $C_1$ does not depend on $t_0$.
\end{remark}

Proposition \ref{propofsing}, follows from a time dependent observability inequality for the standard wave equation. First, a finite time analog of Assumption \ref{TGCCp}.
\begin{assumption}\label{as2}
Fix $T>0$ and consider $Q,$ an open set in $\Omegab \times (0,T)$. Assume for all unit speed generalized geodesics $\gamma,$ there exists $t \in (0,T),$ such that $(\gamma(t),t) \in Q$.
\end{assumption}
Consider the standard wave equation
\begin{equation}\label{WEb}
\begin{cases}
(\p_t^2 + A) \psi =0,&(x,t) \in \Omega \times (0,\infty)\\
B\psi = 0 &\text{when } \p\Omega \neq \emptyset  \\
(\psi,\p_t \psi)|_{t=0} = (\psi_0, \psi_1) &\in H \times L^2(\Omega).
\end{cases}
\end{equation}
\begin{proposition}\label{observeprop}
Suppose $Q$ satisfies Assumption \ref{as2} and let $\chi_Q$ be the indicator function on $Q$. If $\p\Omega \neq \emptyset,$ assume moreover that no generalized geodesic has contact of infinite order with $\p\Omega \times (0,T)$, that is $G^{\infty}=\emptyset$. Then there exists $C_2>0,$ such that for all  $\psi$ solving \eqref{WEb} then
\begin{equation}\label{eq:observe}
\frac{1}{2} \left( \ltwo{\nabla \psi_0}^2 + \ltwo{\psi_1}^2\right) = E(\psi,0) \leq C_2 \ltwooo{\chi_Q \p_t \psi}^2  . 
\end{equation}
\end{proposition}
Proposition \ref{observeprop} is proved in Section \ref{observesection}. The standard term for \eqref{eq:observe} is an observability inequality. 
\begin{remark}
When $\Omega$ has a boundary with Dirichlet condition this follows immediately from Theorem 1.8 in \cite{LRLTT}. When $\p \Omega =\emptyset,$ or the boundary condition is Neumann, this is close to their result, but there is a distinction in the allowed initial data. In particular their result requires $\int \psi_0 dx = \int \psi_1 dx =0$, in that paper see page 5 and their definition of $L^2_0(\Omega)$. It is possible to write a solution $\psi$ of \eqref{WEb}, as $\psi=\vphi+a+bt$, such that $\vphi$ also solves \eqref{WEb} and $\vphi, \p_t \vphi$ have zero mean. However, when plugging $\psi=\vphi+a+bt$ into \eqref{eq:observe}, notice $|\chi \p_t \psi|^2 = \chi^2 |\p_t \vphi|^2+2 b\chi \vphi + \chi^2 b^2$ and $\chi_Q \p_t \vphi$ is not necessarily orthogonal to constants, so $2b\chi\vphi$ has an indefinite sign.
\end{remark}
\begin{remark}
If $W \in C^0_u(\Omegab \times [0,\infty))$ satisfies Assumption \ref{TGCCp} with some $T_0$, then there exists $\d>0$, such that $\{W>\d\}$ satisfies Assumption \ref{as2} on $\Omega \times [0,T_0]$
\end{remark}

With Proposition \ref{propofsing}, it is possible to prove the main result. 
\begin{proof}[Proof of Theorem \ref{mainresult}]
Recall
$$
\frac{d}{dt} E(u,t) = - \int_\Omega W(x,t) |\p_t u|^2 dx_g. 
$$
Integrating in $t$ from $t_0$ to $t_0+T$
$$
E(u,t_0+T) - E(u, t_0) = - \ltwoo{W^{\frac{1}{2}} \p_t u}^2.
$$
Applying Proposition \ref{propofsing}
$$
E(u,t_0+T) \leq (1-C_1) E(u,t_0),
$$
Since $C_1$ is uniform for all $t_0$, for $k \in \Nb$
$$
E(u,t_0+kT) \leq (1-C_1)^k E(u,t_0).
$$
This along with the fact that the energy is non-increasing means that there exists $C,c>0,$ such that for all solutions of \eqref{DWE}
$$
E(u, t+t_0) \leq C e^{-ct} E(u,t_0), \quad t,t_0 \geq 0,
$$
which completes the proof of Theorem \ref{mainresult}.
\end{proof}

\section{Proof of Proposition \ref{propofsing}}\label{propsection}
The approach is to convert an observability inequality for the standard wave equation to an energy bound for the damped wave equation. See \cite{Haraux1989} for an analogous argument when the damping does not depend on time. 

To begin, the following lemma connects observability for the damped wave equation to observability for the standard wave equation with the same initial data, when the observability operator is the damping. 
The exact statement used here is \cite[Lemma 3.3]{PaunonenSeifert2019}.
\begin{lemma}\label{wedwelemma}
Let $(u_0, u_1) \in H \times L^2(M)$. Suppose $W \in L^{\infty}(M \times [0,T])$, $u$ solves \eqref{DWE} and $\psi$ solves \eqref{WEb} with 
$$
(u, \p_t u)|_{t=0}= (\psi, \p_t \psi)|_{t=0}=(u_0, u_1).
$$
Then
$$
\ltwoo{ W^{1/2} \p_t \psi} \leq \left(1 + T \nm{W}_{L^{\infty}_0}\right)^2 \ltwoo{W^{1/2} \p_t u}.
$$
\end{lemma}
With this lemma, Proposition \ref{propofsing} can be proved via a contradiction argument. 
\begin{proof}[Proof of Proposition \ref{propofsing}]
Assume the desired conclusion does not hold, so there exist sequences $t_j \in [0,\infty)$ and $u_j \in L^2(0, \infty; H)$ solving \eqref{DWE}, with $\p_t u_j \in L^2(0, \infty; L^2(\Omega))$ and 
$$
E(u_j, t_j) =1, \qquad \limj \ltwooj{W^{1/2} \p_t u_j} = 0.
$$

Then let $v_j(x,t) = u_j(x,t+t_j)$ and $W_j(x,t) = W(x,t+t_j),$ so $v_j$ solves 
$$
\begin{cases}
(\p_t^2 -\Delta + W_j \p_t) v_j = 0 \\
(v_j, \p_t v_j)|_{t=0} =: (v_{0,j}, v_{1,j}),
\end{cases}
$$
and 
\begin{equation}\label{contra}
E(v_j, 0) =1, \qquad \limj \ltwooo{W_j^{1/2}  \p_t v_j} = 0.
\end{equation}
Note that $\{W_j\}$ forms a pointwise bounded family, since $W \in L^{\infty}(\Omegab \times [0,\infty))$, and $\{W_j\}$ is an equicontinuous family in $C(\Omegab \times [0,T])$, by the uniform continuity of $W$ on $\Omegab \times [0,\infty)$. Therefore by Arzel\`a-Ascoli \cite[Theorem 7.25] {BabyRudin} there exists $W_{\infty} \in C(\Omegab \times [0,T])$ such that, after potentially replacing $W_j$ by a subsequence, $W_j \ra W_{\infty}$ in $L^{\infty}(\Omegab \times [0,T])$. 

Recall $\Cm$ from Assumption \ref{TGCC}. Now, the claim is that $\left\{\Winf>\frac{\Cm}{2}\right\}$ satisfies Assumption \ref{as2}. To see this choose $J$ big enough so that $\|\Winf-W_j\|_{L^{\infty}_0} < \frac{\Cm}{2}$ for $j \geq J$. For any generalized geodesic $\gamma(t)$, by Assumption \ref{TGCCp}
\begin{align*}
\frac{1}{T} \int_0^T \Winf(\gamma(t),t)dt &= \frac{1}{T} \int_0^T W_j(\gamma(t),t) dt + \frac{1}{T} \int_0^T \Winf(\gamma(t),t)-W_j(\gamma(t),t)dt\\
&\geq \frac{1}{T} \int_0^T W(\gamma(t),t+t_j) dt - \frac{1}{T} \int_0^T \|\Winf-W_j\|_{L^{\infty}_0}  dt \\
&\geq \Cm - \frac{\Cm}{2}=\frac{\Cm}{2}.
\end{align*}
Because the average of $\Winf$ on $(0,T)$ is at least $\frac{\Cm}{2}$, $\Winf$ must be at least $\frac{\Cm}{2}$ at some point $(\gamma(t),t)$ for each generalized geodesic $\gamma(t)$, so $\{W_{\infty}>\frac{\Cm}{2}\}$ satisfies Assumption 2. 

Now let $\psi_j$ solve \eqref{WEb} with $(\psi_j, \p_t \psi_j)|_{t=0} = (v_{0,j}, v_{1,j})$. Then by the observability inequality, Proposition \ref{observeprop}, there exists $C_2>0$ such that 
$$
E(v_j, 0) = E(\psi_j, 0) \leq C_2 \ltwooo{W_{\infty}^{1/2} \p_t \psi_j}^2.
$$
Note that this $C_2$ is uniform in $j$ because the observation set $\{W_{\infty}>\frac{\Cm}{2}\}$ does not change. 
To make use of Lemma \ref{wedwelemma}, the $W_{\infty}$ on the right hand side must be replaced by $W_j$. To do so
\begin{align}\label{wjalign}
E(v_j, 0) & \leq C_2  \ltwooo{W_{\infty}^{1/2} \p_t \psi_j}^2 \nonumber \\
&= C_2  \ltwooo{(W_j -W_j+W_{\infty})^{1/2} \p_t \psi_j}^2 \nonumber \\
& \leq C_2 \ltwooo{W_j^{1/2} \p_t \psi_j}^2 + C_2 \lp{W_j-W_{\infty}}{\infty}  \ltwooo{\p_t \psi_j}^2.
\end{align}
Recall $E(\psi_j,0)=E(\psi_j,t)$ for all $t$. Therefore 
$$
\ltwooo{\p_t \psi_j}^2 \leq \int_0^T E(\psi_j,t) dt = T E(\psi_j,0) =T E(v_j, 0).
$$ 
Now choosing $J$ large enough so that $\lp{W_j - W_{\infty}}{\infty} < \left(2TC_2\right)^{-1}$ for $j  \geq J$ then
$$
C_2 \lp{W_j - W_{\infty}}{\infty}\ltwooo{\p_t \psi_j}^2 \leq \frac{1}{2}E(v_j,0).
$$
This term can be absorbed back into the left hand side of \eqref{wjalign} to give
$$
E(v_j,0) \leq 2 C_2 \ltwooo{W_j^{1/2} \p_t \psi_j}^2. 
$$
Now, by Lemma \ref{wedwelemma}
\begin{align*}
E(v_j,0) &\leq 2C_2 \left(1+T \nm{W_j}_{L_0^{\infty}}\right)^2 \ltwooo{W_j^{1/2} \p_t v_j}^2\\
& \leq  2C_2 \left(1+T\nm{W}_{L^{\infty}(\Omega \times [0,\infty))} \right)^2 \ltwooo{W_j^{1/2} \p_t v_j}^2.
\end{align*}
By the second part of \eqref{contra}, the term on the right hand side goes to $0$ as $j \ra \infty$, therefore $E(v_j,0) \ra 0$ as $j \ra \infty$. This contradicts the first part of \eqref{contra}, $E(v_j, 0)=1$, so the desired conclusion must hold. 
\end{proof}

\begin{remark}
Although the contradiction argument concludes with an inequality exactly matching the form of Proposition \ref{propofsing}, the proof cannot be easily rewritten to proceed directly. This is because the observability constant $C_2$ is only uniform by virtue of the contradiction argument. Proceeding directly from Proposition \ref{observeprop} does not work because the observability constant $C_2$ depends on the behavior of $W$ on $[t_0, t_0+T]$ and thus may change as $t_0$ changes.  
\end{remark}

\section{Proof of Proposition \ref{observeprop}}\label{observesection}
The proof of this proposition follows the standard approach for observability inequalities \cite{BardosLebeauRauch1992}, \cite{BurqGerard1997}, \cite{LRLTT}. The idea is to prove a weak version of the observability inequality, that includes some error term, and then eliminate that error term by showing there are no solutions ``invisible" to the observation function $\chi_Q$.

When the boundary conditions are Dirichlet this follows immediately from Theorem 1.8 in \cite{LRLTT}, so throughout this section it is assumed that the boundary condition is Neumann or the boundary is empty. In particular $H$ is always $H^1(\Omega)$ and so is written as such. Also define $H^{-1}(\Omega)$ as the dual of $H^1(\Omega)$. 

Before proceeding, a standard fact. The proof is delayed to the end of the section for readability. 
\begin{lemma}\label{weakconverge}
Suppose $(\psi_{0,n}, \psi_{1,n}) \in H^1(\Omega) \times L^2(\Omega)$ has a weak limit $(\psi_0, \psi_1)$ in $H^1(\Omega) \times L^2(\Omega)$. If $\psi_n$ solves \eqref{WEb} with initial data $(\psi_{0,n}, \psi_{1,n})$ and $\psi$ solves \eqref{WEb} with initial data $(\psi_0,\psi_1)$, then $\psi_n \rhu \psi$ weakly in $L^2(0,T; H^1(\Omega))$ and $\p_t \psi_n \rhu \p_t \psi$ weakly in $L^2(0,T;L^2(\Omega))$.
\end{lemma}

\begin{lemma}[Weak Observability Inequality]\label{weakobserve}
Suppose $Q$ satisfies Assumption \ref{as2}. Then there exists $C>0$, such that for all $(\psi_0, \psi_1) \in H^1(\Omega) \times L^2(\Omega)$ and $\psi$ solving \eqref{WEb} with initial data $(\psi_0, \psi_1)$, then
$$
\hp{\psi_0}{1}^2 + \ltwo{\psi_1}^2 \leq C \left( \ltwooo{\chi_Q \p_t \psi}^2 + \ltwo{\psi_0}^2 + \hp{\psi_1}{-1}^2 \right).
$$
\end{lemma}
\begin{proof}
Assume otherwise, so there exists a sequence $(\psi_{0,n}, \psi_{1,n}) \in H^1(\Omega) \times L^2(\Omega)$, such that 
\begin{align}
\hp{\psi_{0,n}}{1}^2 + \ltwo{\psi_{1,n}}^2 &=1 \label{defectcontra}\\
\limn \ltwo{\psi_{0,n}}^2 + \hp{\psi_{1,n}}{-1}^2 &= 0 \label{defectlone}\\
\limn \ltwooo{\chi_Q \p_t \psi_n}^2 &=0. \label{defectdamping}
\end{align}
Note $(\psi_{0,n}, \psi_{1,n})$ is bounded in $H^1(\Omega) \times L^2(\Omega)$, so it contains a weakly convergent subsequence. By \eqref{defectlone} and the compact embedding of $H^1(\Omega) \times L^2(\Omega)$ into $L^2(\Omega) \times H^{-1}(\Omega)$, the weak limit can only be $(0,0)$. Consider the sequence $\{\psi_n\},$ of solutions to the wave equation with initial data $(\psi_{0,n}, \psi_{1,n})$. By Lemma \ref{weakconverge}, $\psi_n$ weakly converges to $0$ in $H^1(\Omega \times (0,T))$.

Now by Appendix \ref{defect}, up to replacement of $\psi_n$ by a subsequence, there exists a microlocal defect measure $\mu$ on $S^*\hat{\Sigma},$ such that for every $R \in \Psi_b^0(\Omega \times (0,T))$
$$
\<R\psi_n, \psi_n\>_{H^1(\Omega \times (0,T))} \ra \int_{S^*(\Omega \times (0,T))} \kappa(R) d \mu,
$$
where $\kappa(R)$ is the compressed principal symbol of $R$. See Appendix \ref{defect} for details on $\Psi_b^0$ and $\kappa$. By \eqref{defectdamping}, $\mu$ vanishes on $j(T^* Q) \cap S^* \hat{\Sigma}$. Note since $G^{\infty} = \emptyset$, $\mu$ is invariant under the compressed generalized bicharacteristic flow by \cite[Lemma 2.1]{LRLTT} \cite[Section 3]{BurqLebeau2001},\cite[Section 2.2]{Lebeau1996}. The definition of the flow is given in Appendix \ref{nullbichar}. Then by Assumption \ref{as2}, $\mu$ vanishes identically on $S^* \hat{\Sigma}$. Therefore $u_n$ converges strongly to $0$ in $H^1(\Omega \times (0,T))$. Then
$$
0 = \limn \int_0^T \ltwo{\nabla \psi_n}^2 + \ltwo{\psi_n}^2 + \ltwo{\p_t \psi_n}^2  dt \geq \limn \int_0^T E(\psi_n, t) dt.
$$
Since $\psi_n$ is a solution of the wave equation $E(\psi_n,t)=E(\psi_n,0)$ for all $t$ and so 
$$
\limn E(\psi_n, 0) =\limn \frac{1}{2} \left( \ltwo{\nabla \psi_{0,n}} + \ltwo{\psi_{1,n}} \right) = 0, 
$$
this along with \eqref{defectlone} contradicts \eqref{defectcontra}. 
\end{proof}

Now define the set of invisible solutions 
$$
N_T=\{v \in H^1(\Omega \times (0,T)); \Box v =0, (v,\p_t v)|_{t=0}= (v_0, v_1) \in H^1(\Omega) \times L^2(\Omega) \text{ and } \chi_Q \p_t v=0\}.
$$
Equip it with the norm
$$
\nm{v}_{N_T}^2 = \hp{v_0}{1}^2 + \ltwo{v_1}^2,
$$
and note that if $u=v$ in $N_T$, then both $u$ and $v$ solve the wave equation with the same initial data and so $u=v$ almost everywhere on $(0,T) \times \Omega$. 

\begin{lemma}\label{invisiblelemma}
$N_T=\{c\}$, the constant functions. 
\end{lemma}
\begin{proof}
To begin note that for all $v \in N_T$, by the weak observability inequality 
$$
\nm{v}_{N_T} =\hp{v_0}{1}^2 + \ltwo{v_1}^2 \leq C \left( \ltwo{v_0}^2 + \hp{v_1}{-1}^2\right).
$$
By Rellich-Kondrachov, $H^1(\Omega) \times L^2(\Omega)$ is compactly embedded in $L^2(\Omega) \times H^{-1}(\Omega)$, this along with the above inequality implies that the unit ball in $N_T$ is compact and so $N_T$ is finite dimensional. 

Now, if $v \in N_T$, the claim is that $\p_t v \in N_T$ as well. First, since $\chi_Q \p_t v =0$ and every geodesic passes through $Q$, by propagation of singularities, \cite{MelroseSjostrand1978,ms1982}, $v$ is smooth in $\Omega \times (0,T)$. Therefore $\p_t v \in H^1(\Omega \times (0,T) )$ and $\p_t v|_{t=0} \in H^1(\Omega)$ and $\p_t^2 v|_{t=0} \in L^2(\Omega)$. It is immediate that $\Box \p_t v=0$. Finally, since $\chi_Q \p_t v=0$, $\p_t v$ is constant on the open set $Q$ and so $\chi_Q \p_t^2 v=0$ as well. 

Now quotient out by the constant functions 
$$
N_T/\{c\} = \{ [v]; v \in N_T, v_1 \sim v_2 \text{ if } v_1+c=v_2 \text{ for some constant } c\}.
$$
Since $N_T$ is finite dimensional, and $\{c\}$ is a subspace, then $N_T/\{c\}$ is also finite dimensional. Note also that $\p_t$ maps $N_T/\{c\}$ to itself. 

Now to proceed with the proof, assume $N_T \neq \{c\}$ and a contradiction will be produced. Since $N_T/\{c\}$ is finite dimensional and nonzero, $\p_t: N_T/\{c\} \ra N_T /\{c\}$ has at least one eigenvalue $\lambda$ associated to a nontrivial eigenfunction $v$. 

The first claim is that $\lambda \neq 0$. If $v \in N_T /\{c\}$ has $\p_t v=0$, then $\p_t v=c$ in $N_T$. But since $v \in N_T$, then $\chi_Q \p_t v=0$ so $\p_t v=0$. Then $v(x,t)=v(x)$ and $\Box v=0$ implies $-\Delta v=0$, so $v=c$ in $N_T$ and thus $v=0$ in $N_T/\{c\}$. So indeed $\lambda \neq 0$. 

Consider $\p_t v= \lambda v$ in $N_T/\{c\}$, so $\p_t v = \lambda v + c$ in $N_T$ for some constant $c$. Thus for some $q(x)$ 
$$
v(x,t) = e^{\lambda t}q(x) - \frac{c}{\lambda}.
$$ 
Since $(\p_t^2 -\Delta)v(x,t) =0,$ then $(\lambda^2-\Delta) q(x)=0$.

Now take any $t \in (0,T),$ such that $\omega(t) = \{x; (x,t) \in Q\}$ contains a nonempty open set. Since $\chi_Q \p_t v =0,$ then $\chi_Q q(x)=0,$ so $q(x)=0$ on the open set $\omega(t)$. Then by elliptic unique continuation, $q \equiv 0$ on $\Omega$ and so $v\equiv 0$ which is a contradiction. Thus $N_T =\{c\}$ as desired. 
\end{proof}

To complete the proof of Proposition \ref{observeprop} it remains to eliminate the error term.
\begin{proof}[Proof of Proposition \ref{observeprop}]
To begin, the inequality will be shown for solutions with initial position data having average value 0 and then the result will be extended to general initial data. So to begin assume $\int_{\Omega} \psi_0 dx=0$, and \eqref{eq:observe} will be shown. Notice that $\int_{\Omega} \psi_1 dx$ need not be $0$ and it is this case that cannot be handled directly from the existing result.

The argument proceeds by contradiction, so assume there exists a sequence $(\psi_{0,n}, \psi_{1,n}) \in H^1(\Omega) \times L^2(\Omega)$ with $\int_\Omega \psi_{0,n} dx =0$, such that 
\begin{equation}\label{defecterrorcontra}
\ltwo{\nabla \psi_{0,n}}^2+\ltwo{\psi_{1,n}}^2 =1, \qquad \limn \ltwooo{\chi_Q \p_t \psi_n} = 0.
\end{equation}
where $\psi_n$ is the solution of \eqref{WEb} with initial data $(\psi_n, \p_t \psi_n)|_{t=0}=(\psi_{0,n}, \psi_{1,n})$. 

The sequence $(\psi_{0,n},\psi_{1,n})$ is bounded in $H^1(\Omega) \times L^2(\Omega)$, so there exists a weakly convergent subsequence with limit $(\psi_0,\psi_1) \in H^1(\Omega) \times L^2(\Omega)$. Let $\psi$ solve \eqref{WEb}, with initial data $(\psi_0,\psi_1)$. Then by Lemma \ref{weakconverge}, $\p_t \psi_n \rhu \p_t \psi$ weakly in $L^2(\Omega \times (0,T))$. Therefore $\chi_Q \p_t \psi_n \rhu \chi_Q \p_t \psi$ weakly in $L^2(\Omega \times (0,T))$ and so
$$
\ltwooo{\chi_Q \p_t \psi} \leq \liminf_{n \ra \infty} \ltwooo{\chi_Q \p_t \psi} =0.
$$
Thus $\psi \in N_T$ and by Lemma \ref{invisiblelemma}, $\psi=c$, some constant $c$. So $(\psi_0, \psi_1)=(c,0)$ and $(\psi_{0,n}, \psi_{1,n})$ converges to $(c,0)$ weakly in $H^1(\Omega) \times L^2(\Omega)$.  By Rellich-Kondrachov, this convergence is strong in $L^2(\Omega) \times H^{-1}(\Omega)$. But since $\int_\Omega \psi_{0,n} dx=0$ for all $n$, then $c=0$, so $(\psi_{0,n}, \psi_{1,n})$ strongly converges to $(0,0)$ in $L^2(\Omega) \times H^{-1}(\Omega)$. Put another way 
$$
\limn \lp{\psi_{0,n}}{2} + \nm{\psi_{1,n}}_{H^{-1}} = 0. 
$$
Now combining this with Lemma \ref{weakobserve} 
$$
\ltwo{\nabla \psi_{0,n}}^2+\ltwo{\psi_{1,n}}^2 \leq \hp{\psi_{0,n}}{1}^2 + \ltwo{\psi_{1,n}}^2 
\leq C \left( \ltwooo{\chi_Q \p_t \psi_n}^2 +  \lp{\psi_{0,n}}{2} + \nm{\psi_{1,n}}_{H^{-1}} \right).
$$
By \eqref{defecterrorcontra} the left hand side equals 1 for all $n$, while the right hand side goes to $0$ as $n \ra \infty$, which is a contradiction. So \eqref{eq:observe} holds when $\int_\Omega \psi_0 dx=0$. 

When 
$$
\frac{1}{\text{vol}(\Omega)}\int_\Omega \psi_0 dx =\overline{\psi_0} \neq 0,
$$
note that if $\psi(x,t)$ solves $\Box \psi=0$ with initial data $(\psi_0,\psi_1)$, then  $\ti{\psi}(x,t)=\psi(x,t)-\overline{\psi_0}$ solves $\Box \ti{\psi}=0$ with initial data $(\psi_0-\overline{\psi_0}, \psi_1)$. The observability inequality can be applied to $\ti{\psi}$ and the constant $\overline{\psi_0}$ drops out due to the derivatives, proving \eqref{eq:observe} for general initial data.
\end{proof} 

It remains to prove the fact about weak convergence of solutions to the wave equation. 
\begin{proof}[Proof of Lemma \ref{weakconverge}]
Throughout the proof, inner products are taken over $L^2(\Omega)$. First, note $E(\psi_n, t)=E(\psi_n,0)$ for all $t$. Now
\begin{align*}
\p_t \ltwom{\psi_n(t)}^2 = 2 \int_\Omega \psi_n \p_t \psi_n dx &\leq \left( \ltwo{\p_t \psi_n(t)}^2 +  \ltwom{\psi_n(t)}^2 \right)\\
&\leq \left( E(\psi_n, t) + \ltwom{\psi_n(t)}^2\right).
\end{align*}
Therefore by Gr\"onwall's inequality
$$
\ltwom{\psi_n(t)}^2 \leq e^{t} \left( \ltwom{\psi_n(0)}^2 + \int_0^t E(\psi_n, s) ds \right). 
$$
So
\begin{align*}
\max_{0 \leq t \leq T} \left( \ltwom{\psi_n(t)}^2 + E(\psi_n, t) \right) &\leq (Te^T +1) \left( E(\psi_n, 0) + \ltwom{\psi_n(0)}^2 \right) \\
&= C \left( \hp{\psi_{0,n}}{1}^2 + \ltwo{\psi_{1,n}}^2 \right).
\end{align*}
Furthermore for any $v \in H^1(\Omega)$ with $\hp{v}{1} \leq 1$
$$
\< \p_t^2 \psi_n(t), v(t)\> + \<\nabla \psi_n(t), \nabla v(t)\> =0, 
$$
so 
$$
|\<\p_t^2 \psi_n(t), v(t)\>| \leq \ltwo{\nabla \psi_n(t)} \leq E(\psi_n,t)^{1/2}.
$$
And thus 
$$
\int_0^T \hp{\p_t^2 \psi_n(t) }{-1}^2 dt \leq T E(\psi_n, 0). 
$$
Thus there exists $C>0$ such that 
\begin{equation}
\max_{0 \leq t \leq T} \left( \ltwom{\psi_n(t)}^2 + E(\psi_n, t) \right) + \int_0^T \hp{\p_t^2 \psi_n(t)}{-1}^2 dt \leq C \left( \hp{\psi_{0,n}}{1}^2 + \ltwo{\psi_{1,n}}^2 \right).
\end{equation}
Therefore $\psi_n$ is bounded in $L^2(0,T; H^1(\Omega)), \p_t \psi_n$ is bounded in $L^2(0,T; L^2(\Omega))$ and $\p_t^2 \psi_n$ is bounded in $L^2(0,T;H^{-1}(\Omega))$. 

Thus there exists $v \in L^2(0,T; H^1(\Omega))$ such that, up to replacement by subsequences, $\psi_n \rhu v$ weakly in $L^2(0,T; H^1(\Omega)),$ $\p_t \psi_n \rhu \p_t v$ weakly in $L^2(0,T;L^2(\Omega))$ and $\p_t^2 \psi_n \rhu \p_t^2 v$ weakly in $L^2(0,T;H^{-1}(\Omega))$. The proof will be completed if it is shown that $v$ solves the wave equation and has $(v,\p_t v)|_{t=0} = (\psi_0, \psi_1)$. 

To see this let $g \in C^1([0,T]; H^1(\Omega))$. Then since $\psi_n$ solves \eqref{WEb} 
\begin{equation}\label{weakwave}
\int_0^T \< \p_t^2 \psi_n, g\> + \< \nabla \psi_n, \nabla g\> dt =0.
\end{equation}
And so in the limit as $n \ra \infty$, by the weak convergence of $\p_t^2 \psi_n$ to $\p_t^2 v$ and $\nabla \psi_n$ to $\nabla v$
$$
\int_0^T \<\p_t^2 v, g\> + \<\nabla v, \nabla g\> dt =0.
$$
That is $v$ solves \eqref{WEb}. Note also $v \in C([0,T]; L^2(\Omega)), \p_t v \in C([0,T]; H^{-1}(\Omega))$, so it makes sense to evaluate $v, \p_t v$ at $t=0$. Now choose $f \in C^2([0,T]; H^1(\Omega))$ with $f(T)=\p_t f(T)=0$. Then replacing $g$ by $f$ in \eqref{weakwave} and integrating by parts twice in $t$
$$
 \int_0^T \<\psi_n, \p_t^2 f\> + \<\nabla \psi_n, \nabla f\> dt = \< \p_t \psi_n(0), f(0)\> - \<\psi_n(0), \p_t f(0)\>.
$$
Similarly 
$$
 \int_0^T \<v, \p_t^2 f\> + \<\nabla v, \nabla f\> dt = \< \p_t v(0), f(0)\> - \<v(0), \p_t f(0)\>.
 $$
 By weak convergence of $\psi_n$ to $v$ in $L^2(0,T;H^1(\Omega))$ the left hand sides of the two preceding equations are equal after taking the limit as $n \ra \infty$. So
 $$
 \limn \< \p_t \psi_n(0), f(0)\> - \<\psi_n(0), \p_t f(0)\> = \< \p_t v(0), f(0)\> - \<v(0), \p_t f(0)\>.
 $$
 Now note $\psi_n(0) =\psi_{0,n} \rhu \psi_0$ in $H^1(\Omega)$ and $\p_t \psi_n(0) = \psi_{1,n} \rhu \psi_1$ in $L^2(\Omega)$ so
 $$
 \< \psi_1, f(0)\> - \<\psi_0, \p_t f(0)\> = \< \p_t v(0), f(0)\> - \<v(0), \p_t f(0)\>.
 $$
 Since $f(t=0), \p_t f(t=0)$ are arbitrary $(v,\p_t v)|_{t=0}=(\psi_0,\psi_1)$. Therefore $\psi=v$, by uniqueness of weak solutions of the wave equation. The desired convergence was already shown. 
\end{proof}

\appendix 
\section{Appendix A}
\subsection{Generalized Null Bicharacteristics}\label{nullbichar}
This appendix introduces the generalized bicharacteristic flow of \cite{MelroseSjostrand1978}, see also \cite[Chapter 24]{Hormander3}. The exposition is adapted from Section 1 of \cite{LRLTT}.

Let $g^*$ be the dual metric to $g$. On $T^*(M \times \Rb),$ the principal symbol of $\p_t^2 -\Delta_g$ is $p(x,t,\xi, \tau) = -\tau^2 + g_x^*(\xi,\xi)$, where $(\tau, \xi)$ are the fiber variables for $(t,x)$. The Hamilton vector field of $p$ is given by $H_p f = \{p,f\}$. Classical bicharacteristics are the integral curves of $H_p$ in $\Char(p) = \{p=0\}$. For $\p_t^2-\Delta_g$, the projection of classical bicharacteristics onto $M$, using $t$ as a parameter, are exactly unit speed geodesics on $M$. 

Define $Y = \Rb \times \Omegab$ and $\Char_Y(p) = \{\rho=(x,t,\xi,\tau) \in T^*(M \times \Rb)\backslash 0; \, x \in \Omegab \text{ and } p(\rho)=0\}$. Let $M$ have dimension $d$. Close to the boundary of $\Omega$ use geodesic normal coordinates $(x', x_d)$ where $x'=(x_1, x_2, \ldots, x_{d-1})$. So $x_d=0$ at $\p\Omega$ and $x_d>0$ in $\mathring{\Omega}$. Set $y=(t,x), y'=(t,x')$ and $y_n=x_d$ where $n=d+1$, which provide coordinates near $\p Y= \p \Omega \times \Rb$. Let $\eta=(\eta', \eta_n)$ be the cotangent variables associated to $y=(y',y_n)$. In these coordinates, the principal symbol of the wave operator is
$$
p(y', y_n, \eta', \eta_n) = \eta_n^2 + r(y, \eta'),
$$
where $r$ is a smooth $y_n$-family of tangential differential symbols.  Define $T^* Y= T^* (\Omega \times \Rb)$ and its boundary
$$
\p T^* Y = \{ \rho = (y, \eta) \in T^*(\Rb \times \Omega); y_n =0\}.
$$
Use $r_0$ to denote the restriction of $r$ to $\p T^* Y$, that is, $r_0(y', \eta') = r(y', y_n=0, \eta')$. 

Define $\Sigma_0 = \{ \rho=(y',y_n,\eta', \eta_n) \in \Char_Y(p); y_n=0\}$ or $\Sigma_0 = \Char_Y(p) \cap \p T^* Y$. Using local coordinates define the glancing set $G \subset \Sigma_0$ by 
$$
G= \{(y', y_n=0, \eta', \eta_n) \in \Sigma_0; \eta_n=0\}.
$$
The glancing set can also be decomposed as $G=G^2 \supset G^3 \supset \cdots \supset G^{\infty}$, with $\rho=(y', y_n=0, \eta', \eta_n) \in G^{k+2}$ when 
$$
\eta_n=r_0(y',\eta' )=0, \quad H_p^j y_n =0, \quad 0 \leq j <k.
$$
Finally, consider $G^2 \backslash G^3$, the glancing set of order precisely 2, and define subsets of it, the diffractive set $G_d^2$ and the gliding set $G_g^2$, as 
$$
\rho \in G^2_d \, \,(\text{resp. } G_g^2) \iff \rho \in G^2 \backslash G^3 \,  \text{ and } \, H_p^2 y_n< 0 \, \,(\text{resp.} >0).
$$
Then $G^2 \backslash G^3 = G_d^2 \cup G_g^2$. This decomposition can be continued for higher even orders of glancing points, but is not needed in this paper. 
\begin{definition}\label{genbichar}
A generalized bicharacteristic of $p$ is a differentiable map 
$$
\Rb \backslash \mathcal{B} \ni s \mapsto \gamma(s) \in (\Char_Y(p) \backslash \Sigma_0) \cup G,
$$
satisfying the following properties
\begin{enumerate}
	\item $\gamma'(s) = H_p(\gamma(s))$ if $\gamma(s) \in \Char_Y(p) \backslash \Sigma_0$ or $\gamma(s) \in G^2_d$ 
	\item $\gamma'(s) = H_{r_0}(\gamma(s))$ if $\gamma(s) \in G \backslash G^2_d$. 
	\item  Every $s_0 \in \mathcal{B}$ is isolated, and there exists $\d>0$, such that for $s \in (s_0-\d, s_0) \cup (s_0, s_0+\d)$ then $\gamma(s) \in \Char_Y(p) \backslash \Sigma_0$. Furthermore, the limits $\lim_{s \ra s_0^{\pm}} \gamma(s)=(y^{\pm}, \eta^{\pm})$ exist and $y_n^-=y_n^+ =0, y^-{}'=y^+{}', \eta^-{}'=\eta^+{}',$ and $\eta_n^-=-\eta_n^+$. 
\end{enumerate}
\end{definition}
In case 1 the generalized bicharacteristic is either in the interior, or at a diffractive point. Here it coincides with a segment of a classical bicharacteristic. Case 2 describes how a generalized bicharacteristic enters or leave the boundary $\p T^* Y$ or locally remains in it. Case 3 describes reflections, when a bicharacteristic transversally encounters the boundary.  

For $y$ near the boundary of $Y$, define ${}^b T_y Y$ to be the tangent vector field generated by $\p_{y'}$ and $y_n \p_{y_n}$. Then define the compressed cotangent bundle ${}^bT^*Y=\bigcup_{y \in Y} ({}^b T_yY)^*$, and define the compression map
\begin{align*}
&j: T^* Y \ra {}^b T^* Y, \\
&(y, \eta', \eta_n) \mapsto (y, \eta', y_n\eta_n).
\end{align*}
Note that 
\begin{itemize}
	\item for $y \in \Rb \times \Omega,$ then ${}^bT^*_yY=j(T^*_y Y)$ is isomorphic to $T^*_y=T^*_y(\Omega \times \Rb)$,
	\item for $y \in \Rb \times \p \Omega,$ then ${}^b T^*_y Y= j(T^*_y Y)$ is isomorphic to $T^*_y(\p \Omega \times \Rb)$.
\end{itemize}
The set of points $(y', y_n=0, \eta', 0) \in {}^b T^* Y|_{y_n=0}$ such that $r_0(y', \eta')>0$ is called the elliptic set $E$. Also set $\hat{\Sigma} = j(\Char_Y(p)) \cup E$ and take the cosphere quotient space $S^* \hat{\Sigma} = \hat{\Sigma}/(0,+\infty)$. This is needed in the defect measure construction. 

Define compressed generalized bicharacteristics to be the image under $j$ of the generalized bicharacteristics of Definition \ref{genbichar}. If ${}^b \gamma=j(\gamma)$ is a compressed generalized bicharacteristic, then ${}^b \gamma:\Rb \ra {}^b T^* Y \backslash E$ is a continuous map. Using $t$ as a parameter,  projecting compressed generalized bicharacteristics down to $M$ gives unit speed generalized geodesics for $\Omega$. Generalized geodesics remain in $\Omegab$.  In geometric optics the standard terminology for such a projection is ``ray". 

An important feature of compressed generalized bicharacteristics, as shown in \cite{MelroseSjostrand1978}, is the following proposition. 
\begin{proposition}
A compressed generalized bicharacteristic with no point in $G^{\infty}$ is uniquely determined by any one of its points.
\end{proposition}

\subsection{Defect Measure with Boundary}\label{defect}
\begin{definition}
Define $\Psi_b^m(Y)$ to be made up of operators of the form $R=R_{\inter} + R_{\tan}$ where $R_{\inter}$ is a classical pseudodifferential opperator of order $m$, with compact support in $\Rb \times \Omega,$ and $R_{\tan}$ is a classical tangential pseudodifferential operator of order $m$. In the local normal coordinates introduced in Appendix \ref{nullbichar}, $R_{\tan}$ acts only in the $y'$ variables.  
\end{definition}
Let $\sigma(R_{\inter})$ and $\sigma(R_{\tan})$ be the homogeneous principal symbols of $R_{\inter}$ and $R_{\tan}$ respectively. The restrictions $\sigma(R_{\inter})|_{\Char_Y(p)}$ and $\sigma(R_{\tan})|_{\Char_Y(p) \cup T^* (\Rb \times \p \Omega)}$ make sense, and under the compression map $j: T^* Y \ra {}^b T^* Y$ 
$$
j\left(\sigma(R_{\inter})|_{\Char_Y(p)} + \sigma(R_{\tan})|_{\Char_Y(p) \cup T^* (\Rb \times \p \Omega)} \right) = :\kappa(R),
$$ 
is a continuous function on $\hat{\Sigma}= j (\Char_Y(p)) \cup E$. Furthermore, by the homogeneity of the symbols $\kappa(R)$ is a continuous function on $S^* \hat{\Sigma} = \hat{\Sigma} / (0, \infty)$. Then by \cite[Section 2.1]{Lebeau1996} and \cite[Proposition 2.5]{BurqLebeau2001}
\begin{proposition}
Suppose $\{u_n\}$ is a bounded sequence in $H^1(\Omega\times \Rb)$. If $(\p_t^2 -\Delta_g) u_n=0$ and $u_n$ weakly converges to $0$, then there exists a subsequence $\{u_{n_j}\}$ and a positive measure $\mu$ on $S^* \hat{\Sigma}$ such that for any $R \in \Psi^0(Y)$
$$
\<Ru_{n_j}, u_{n_j}\>_{H^1(\Omega \times \Rb)} \ra \<\mu, \kappa(R)\>.
$$ 
\end{proposition}
This is a generalization of \cite{Gerard1991} and \cite{Tartar1990}.
\bibliographystyle{alpha}
\bibliography{mybib}

\end{document}